\documentclass[10pt,a4paper]{article}

\usepackage{a4wide}
\usepackage{amsthm}
\usepackage{amsfonts}
\usepackage{amsmath}
\usepackage[english]{babel}

\usepackage{faktor}
\usepackage{etex}
\usepackage{hyperref}
\usepackage{multimedia}
\usepackage[all]{xy}
\usepackage{amssymb}
\usepackage{graphicx,textpos}
\usepackage[dvipsnames]{xcolor}
\usepackage{helvet}
\usepackage{stmaryrd}
\usepackage{enumerate}
\usepackage{todonotes}
\usepackage{pdfsync}
\usepackage{dsfont}
\usepackage[shortcuts]{extdash}
\usepackage[all]{xy}
  \newdir{ >}{{}*!/-9pt/@{>}}

\newcommand{\N}{\mathbb{N}}
\newcommand{\R}{\mathbb{R}}
\newcommand{\Q}{\mathbb{Q}}

\newcommand{\C}{\mathbb{C}}

\newcommand{\I}{\mathds{1}}

\newcommand{\rst}[1]{\ensuremath{{\mathbin\mid}\raise-.5ex\hbox{$#1$}}}
\newcommand{\lie}{\mathfrak{g}}
\newcommand{\lien}{\mathfrak{n}}
\newcommand{\lieh}{\mathfrak{h}}

\newcommand{\lier}{\mathfrak{r}}

\DeclareMathOperator{\Gal}{Gal}
\DeclareMathOperator{\GL}{GL}
\DeclareMathOperator{\SL}{SL}
\DeclareMathOperator{\Aut}{Aut}

\author{Jonas Der\'e\thanks{The author was supported by a postdoctoral fellowship of the Research Foundation -- Flanders (FWO).}}
\title{\textbf{A note on the structure \\ of underlying Lie algebras}}
\date{\today}

\newtheorem{Def}{Definition}[section]
\newtheorem{Ex}[Def]{Example}
\newtheorem{Cor}[Def]{Corollary}
\newtheorem{Thm}[Def]{Theorem}
\newtheorem{Prop}[Def]{Proposition}
\newtheorem{Lem}[Def]{Lemma}

\newtheorem*{Prop*}{Proposition}
\newtheorem*{Lem*}{Lemma}

\newtheorem{QN}{Question}
\newtheorem{Con}{Conjecture}
\makeatletter
\newtheorem*{rep@theorem}{\rep@title}
\newcommand{\newreptheorem}[2]{%
	\newenvironment{rep#1}[1]{%
		\def\rep@title{#2 \ref{##1}}%
		\begin{rep@theorem}}%
		{\end{rep@theorem}}}
\makeatother

\newreptheorem{Thm}{Theorem}
\newreptheorem{Cor}{Corollary}

\newcommand{\suchthat}{\;\ifnum\currentgrouptype=16 \middle\fi|\;}

\newcommand{\compcent}[1]{\vcenter{\hbox{$#1\circ$}}}

\newcommand{\comp}{\mathbin{\mathchoice
		{\compcent\scriptstyle}{\compcent\scriptstyle}
		{\compcent\scriptscriptstyle}{\compcent\scriptscriptstyle}}}

\begin{document}

\maketitle

\begin{abstract}
Every Lie algebra over a field $E$ gives rise to new Lie algebras over any subfield $F \subseteq E$ by restricting the scalar multiplication. This paper studies the structure of these underlying Lie algebra in relation to the structure of the original Lie algebra, in particular the question how much of the original Lie algebra can be recovered from its underlying Lie algebra over subfields $F$. By introducing the conjugate of a Lie algebra we show that in some specific cases the Lie algebra is completely determined by its underlying Lie algebra. Furthermore we construct examples showing that these assumptions are necessary.

As an application, we give for every positive $n$ an example of a real $2$-step nilpotent Lie algebra which has exactly $n$ different bi-invariant complex structures. This answers an open question by Di Scala, Lauret and Vezzoni motivated by their work on quasi-K\"ahler Chern-flat manifolds in differential geometry.  The methods we develop work for general Lie algebras and for general Galois extensions $F \subseteq E$, in contrast to the original question which only considered nilpotent Lie algebras of nilpotency class $2$ and the field extension $\R \subseteq \C$. We demonstrate this increased generality by characterizing the complex Lie algebras of dimension $\leq 4$ which are defined over $\R$ and over $\Q$. 
\end{abstract}

\section{Introduction}

Given a field extension $F \subseteq E$ and a Lie algebra $\lie$ over the field $E$, we define the underlying Lie algebra $\lie_F$ by regarding $\lie$ as a Lie algebra over only the subfield $F$. Many properties of the underlying Lie algebra correspond to properties of the original Lie algebra, e.g.~the Lie algebra $\lie$ is abelian if and only if the same holds for the Lie algebra $\lie_F$. Similarly, it easily follows that being nilpotent or solvable is equivalent for the Lie algebras $\lie$ and $\lie_F$, where moreover the nilpotency class and the derived length correspond. In this paper we investigate more deeply the algebraic structure of the underlying Lie algebra $\lie_F$ in relation to the Lie algebra $\lie$.

The main question we explore is under which conditions we can completely recover the algebraic structure of the Lie algebra $\lie$ from the underlying Lie algebra $\lie_F$. The motivation comes from a geometric problem of \cite{dlv12-1} in the special case of the field extension $\R \subseteq \C$, where the algebraic counterpart goes as follows.

\begin{QN}
	\label{Q5}
	Do there exist two non-isomorphic complex $2$-step nilpotent Lie algebras $\lie$ and $\lieh$ such that $\lie_\R \approx \lieh_\R$? 
\end{QN}
\noindent The background of this question lies in the study of compact quasi-K\"ahler Chern-flat manifolds. These manifolds are obtained as the quotient of a $2$-step nilpotent Lie group by a cocompact lattice, where the Lie algebra corresponding to the Lie group has a anti-bi-invariant almost complex structure. In \cite{dlv12-1} it is shown that these Lie algebras are exactly the underlying real Lie algebras of $2$-step nilpotent complex Lie algebras with the anti-bi-invariant almost complex structure determined by the bi-invariant complex structure.  

Underlying Lie algebras occur in many other places as a tool to construct new Lie algebras with specific properties and this for arbitrary field extensions $F \subseteq E$ of finite degree. For example, in \cite{bdv18-1} the authors use underlying Lie algebras over field extensions of finite degree to construct almost inner derivations which are not inner, related to the construction of isospectral nilmanifolds which are not isometric \cite{gw84-1}. Another instance comes from the construction of Anosov Lie algebras, where the underlying rational Lie algebra of the Heisenberg Lie algebra $\lieh_3\left(\Q(\sqrt{d})\right)$ with $d \in \N$ a square-free number was the first example of an Anosov Lie algebra which is not abelian, see \cite{smal67-1}. 

These applications indicate that we should not study Question \ref{Q5} for $2$-step nilpotent Lie algebras over the complex numbers $\C$ only. We suggest the following generalized question for field extensions of finite degree.

\begin{QN}
	\label{generalQN}
Let $F \subseteq E$ be a field extension of finite degree and $\lie, \hspace{0.5mm} \lieh$ be Lie algebras over the field $E$. Under which conditions on the Lie algebras $\lie$ and $\lieh$ does $\lie_F \approx \lieh_F$ imply that $$\lie \approx \lieh?$$
\end{QN}

\noindent We assume that the field extension $F \subseteq E$ has finite degree to ensure that the underlying Lie algebra $\lie_F$ has finite dimension if the original Lie algebra $\lie$ is finite-dimensional.

To study Question \ref{generalQN} we introduce the $\sigma$-conjugate $\lie^\sigma$ of a Lie algebra $\lie$ over $E$ for every $\sigma \in \Aut(E,F)$. One way of defining the Lie algebra $\lie^\sigma$ is taking a basis for the Lie algebra $\lie$ and applying the map $\sigma$ to the structure constants for this basis, see Example \ref{conjugatestructureconstants}. Our first main result states that two conjugate Lie algebras have the same underlying Lie algebra over $F$.

\begin{repThm}{sameunder}
Let $F \subseteq E$ be a field extension and $\sigma \in \Gal(E,F)$. If we denote by $\lie^\sigma$ the $\sigma$-conjugate of the Lie algebra $\lie$ over the field $F$, then $$\left(\lie^\sigma \right)_F \approx \lie_F.$$
\end{repThm}

\noindent So in particular, to produce a negative answer on Question \ref{Q5}, it suffices to construct a complex Lie algebra $\lie$ such that $\lie \not \approx \overline{\lie}$, where $\overline{\lie}$ denotes the complex conjugate of the Lie algebra $\lie$. We give $2$-step nilpotent examples in Section \ref{sec:examples} and even give for every $n \in \N$ a real $2$-step nilpotent Lie algebra which is the real underlying Lie algebra of exactly $n$ different complex Lie algebras. This demonstrates the necessity of extra conditions on $\lie$ and $\lieh$ in Question \ref{generalQN}.

The main theorem for studying Question \ref{generalQN} gives the structure of the underlying Lie algebra if we again extend the scalars on $\lie_F$ to the original field $E$ for Galois extensions $F \subseteq E$.
\begin{repThm}{sumconjugate}
	Let $F \subseteq E$ be a Galois extension and let $\lie$ be a Lie algebra over the field $E$. Let $\lie_F$ be the underlying Lie algebra over $F$, then there is an isomorphism $$ \lie_F \otimes_F E \approx \bigoplus_{\sigma \in \Gal(E,F)} \lie^\sigma $$ where $\lie^\sigma$ is the $\sigma$-conjugate of the Lie algebra $\lie$. 
\end{repThm}

We will show that if $\lie$ is defined over $F$, then every conjugate Lie algebra over $F$ is isomorphic to $\lie$, i.e.~for every $\sigma \in \Aut(E,F)$ we have $\lie^\sigma \approx \lie$. This fact also serves as a tool to characterize the Lie algebras which are defined over the field $F$ and we illustrate this for Lie algebras of dimension $\leq 4$ in Section \ref{sec:ex2}.  By using the techniques in \cite{fgh13-1} about indecomposable Lie algebras, we achieve the following solution to Question \ref{generalQN}.

\begin{repCor}{partialanswer}
Let $F \subseteq E$ be a Galois extension and $\lie$ a Lie algebra over $E$ which is defined over the subfield $F$. If $\lieh$ is another Lie algebra over $E$ such that $\lie_F \approx \lieh_F$, then the original Lie algebras $\lie \approx \lieh$ are isomorphic as well if either
\begin{enumerate}
	\item the Lie algebra $\lieh$ is defined over $F$, or
	\item the Lie algebra $\lie$ is indecomposable.
\end{enumerate}

\end{repCor}
\noindent As another consequence, we show that there are up to isomorphism only finitely many Lie algebras having the same underlying Lie algebra, including a way to describe them in Theorem \ref{classification}.

This paper is structured as follows. First we give some preliminaries about field extensions, restricting and extending the scalars of Lie algebras and how to decompose them into indecomposable ideals in Section \ref{sec:prel}. Next, Section \ref{sec:conj} introduces the conjugate of a Lie algebra and describes its properties, including a proof of Theorem \ref{sameunder}. The proof and applications of Theorem \ref{sumconjugate} follow in Section \ref{sec:ans} and we give some examples, both over the complex numbers as over other fields in Section \ref{sec:examples}, including the answer to Question \ref{Q5}. Finally we discuss some open questions about the existence of real forms and conjugate Lie algebras in Section \ref{sec:qu}.

\paragraph*{Acknowledgements}

I would like to express my gratitude to J. Lauret for pointing out the open question in his work \cite{dlv12-1} and K. Dekimpe for introducing me to the results about decomposing Lie algebras in \cite{fgh13-1}. I want to thank B.~Verbeke for showing me the preliminary results of \cite{bdv18-1} about almost inner derivations.

\section{Preliminaries}
\label{sec:prel}

In this section, we introduce the necessary background for the main results and fix notations for the remainder of this paper. First we present the main concepts from Galois theory, afterwards we demonstrate how to construct from a given Lie algebra over a field new Lie algebras over subfields and field extensions by restricting and extending the scalar multiplication respectively. Finally we recall some results from \cite{fgh13-1} about decomposing Lie algebras into indecomposable ideals.

\paragraph{Galois extensions}
Let $F \subseteq E$ be a field extension. The automorphisms $\Aut(E)$ of the field $E$ form a group under the composition which contains the subgroup  $$\Aut(E,F) = \left\{\sigma \in \Aut(E) \suchthat \forall x \in F: \sigma(x) = x \right\}$$ of automorphisms fixing the subfield $F$. If $\Q \subseteq E$ is a field of characteristic $0$, then $\Aut(E,\Q)$ is equal to automorphism group $\Aut(E)$. Let $F \subseteq E$ be a field extension such that $E$ is algebraically closed, then every automorphism $\sigma: F \to F$ extends to an automorphism $\tilde{\sigma}: E \to E$. In particular this property holds for every automorphism $\sigma: F \to F$ of a subfield $F \subseteq \C$. 

The field $E$ always forms a vector space over the field $F$ and we define the degree $[E:F]$ of the field extension by $[E:F] = \dim_F (E)$. We say that $F \subseteq E$ is a field extension of finite degree if $[E:F] < \infty$ and in this case the order of the group $\Aut(E,F)$ satisfies $$\vert \Aut(E,F) \vert \leq [E:F].$$ If equality holds in the latter inequality, we say that $E$ is a \textbf{Galois extension} of $F$, calling $\Gal(E,F) = \Aut(E,F)$ the \textbf{Galois group} of the field extension $F \subseteq E$. Under this assumption the field $F$ is equal to the fixed field of the Galois group, i.e.~$x \in F$ if and only if $\sigma(x) = x$ for all $\sigma \in \Gal(E,F)$. In fact, the latter statement is equivalent to being a Galois extension for finite degree field extensions. 

If $F \subseteq E$ is a field extension of finite degree in characteristic $0$, then there always exists a field extension $E \subseteq E^\prime$ of finite degree such that $F \subseteq E^\prime$ is a Galois extension. Hereinafter we usually assume that field extensions are Galois extensions, although the main results could also be stated for general finite seperable field extensions, albeit more technically.

For fixing notations, we introduce the easiest example of a Galois extension, namely $\R \subseteq \C$.  
\begin{Ex}
	\label{CoverR}
The Galois group $\Gal(\C,\R) = \{\I_\C, \overline{ \phantom{z}}  \}$ consists of two elements, where $\overline{\phantom{z}}: \C \to \C$ is the complex conjugation map. 
\end{Ex} 
\noindent Although this example is important for answering Question \ref{Q5}, we will also apply our main results about underlying Lie algebras to Galois extensions $\Q \subseteq E$ in Section \ref{sec:examples}.

\paragraph{Restricting and extending scalar multiplication}

Given a Lie algebra $\lie$ over the field $E$, there are two different ways in which we can construct a new Lie algebra over a different field, namely by either restricting or extending the scalar multiplication. 

The first method, which we already discussed in the introduction, starts by taking a subfield $F \subseteq E$ to construct the underlying Lie algebra $\lie_F$. If $V$ is a vector space over the field $E$, then the scalar multiplication is given by a map $E \times V \to V$ which satisfies certain conditions for being a vector space. The underlying vector space $V_F$ is defined by the inclusion of the set $F \times V$ in $E \times V$ and thus defining a new scalar multiplication $F \times V \hookrightarrow E \times V \to V$. In this case we say that we \textbf{restricted} the scalar multiplication to the subfield $F$. Note that if $V$ is finite-dimensional and if $F \subseteq E$ is a field extension of finite degree, then $V_F$ is also finite-dimensional with dimension $$\dim_F(V_F) = [E:F] \dim_E(V).$$ If $\lie$ is a Lie algebra over $E$, then the Lie bracket will in particular be $F$-linear, hence $\lie_F$ forms a Lie algebra over $F$ of dimension $\dim_F(\lie_F) = [E:F] \dim_E(\lie)$, which we call the \textbf{underlying Lie algebra} over $F$. In the special case of $\R \subseteq \C$ we say that $\lie$ is given by a \textbf{bi-invariant complex structure} on $\lie_\R$.

The second method is to take a field extension $E \subseteq F^\prime$ and \textbf{extend} the scalar multiplication of the Lie algebra $\lie$ by considering the tensor product $\lie \otimes_E F^\prime$. In some papers, this Lie algebra is denoted by a superindex $F^\prime$, but this notation is avoided because of its similarity to the notation for the underlying Lie algebra. In this case, $$\dim_{F^\prime} \left( \lie \otimes_E F^\prime \right) = \dim_E \left( \lie \right)$$ and we say that $\lie$ is an \textbf{$E$-form} of the Lie algebra $\lie \otimes_E F^\prime$. If a Lie algebra $\lieh$ over $F^\prime$ is isomorphic to $\lie \otimes_E F^\prime$ for some Lie algebra $\lie$ over $E$, we say that $\lieh$ is \textbf{defined over $E$}. 

In some texts the real underlying Lie algebra $\lie_{\R}$ of a complex Lie algebra $\lie$ is called the realification of $\lie$. We avoid this name because of its resemblance to the complexification of a real Lie algebra, which means extending the field to $\C$, in contrast to restricting the field to $\R$. The underlying real Lie algebra has been studied for its importance in studying simple Lie algebras over $\R$, see for example \cite[Page 34]{ov93-2}.

Although this is clear from the definitions, we want to emphasize that for a non-trivial field extension $F \subseteq E$ and a Lie algebra $\lie$ over $E$, it does not hold that $\lie_F \otimes_F E$ is isomorphic to $\lie$, which can already be seen by comparing the dimensions. Up till now, it is not described what the relation is between these two methods of changing the field. The results of this paper can also be regarded as to which Lie algebra you get by first restricting and then again extending the scalar multiplication, see Theorem \ref{sumconjugate}.

\paragraph{Indecomposable Lie algebras}
In this part we recall some results of \cite{fgh13-1} about how to decompose a Lie algebra into indecomposable ideals. Recall that a Lie algebra $\lie$ is called \textbf{decomposable} if there exists two proper ideals $\lie_1, \lie_2 \le \lie$ such that $\lie = \lie_1 \oplus \lie_2$. Note that in this case $[\lie_1,\lie_2] = 0$ since both subspaces are ideals and $\lie_1 \cap \lie_2 = 0$. A Lie algebra is called \textbf{indecomposable} if it is not decomposable into proper ideals. A decomposition of a Lie algebra into indecomposable ideals is a finite set of ideals $\lie_1, \ldots, \lie_k$ such that $$ \lie = \bigoplus_{i=1}^k \lie_i$$ and every ideal $\lie_i$ is indecomposable. 

Although \cite[Theorem 3.3]{fgh13-1} is only formulated for real Lie algebras, the proof works for Lie algebras over any field.

\begin{Thm}
	\label{krull}
	If $\lie$ is a Lie algebra over the field $E$ with two decompositions $$\lie = \bigoplus_{i=1}^k \lie_i = \bigoplus_{j=1}^l \lieh_i$$ into indecomposable ideals $\lie_i$ and $\lieh_j$, then the number of ideals is equal, i.e.~$k = l$, and moreover up to renumbering the ideals it holds that $\lie_i \approx \lieh_i$ for all $1 \leq i \leq k$.
\end{Thm}

The statement in \cite{fgh13-1} is a stronger version which also describes other uniqueness properties of the decomposition, but for our purposes the formulation of Theorem \ref{krull} suffices. Note that uniqueness of the decomposition only holds up to isomorphism, as we illustrate for clarity in the following example.

\begin{Ex}
Let $\lieh_3(E)$ be the Heisenberg Lie algebra of dimension $3$ over the field $E$, i.e.~the Lie algebra with basis $X, Y, Z$ and only non-trivial bracket given by $$[X,Y] = Z.$$ It is easy to see that this Lie algebra is indecomposable since every proper ideal is abelian. 

Consider the direct sum $\lie = \lieh_3(E) \oplus \lieh_3(E)$ with basis $X_i, Y_i, Z_i$ for the $i$-th component of the direct sum. On the one hand, we can decompose $\lie$ as we defined it, namely as the direct sum of the ideals $\langle X_1, Y_1, Z_1 \rangle$ and $\langle X_2, Y_2, Z_2 \rangle$. On the other hand, the ideals $$\langle X_1, Y_1 + Z_2, Z_1 \rangle \text{ and } \langle X_2, Y_2, Z_2\rangle$$ give a second decomposition of $\lie$. Although each of these ideals is isomorphic to $\lieh_3(E)$, the decompositions are not identical.
\end{Ex}
\noindent This example illustrates that the isomorphisms we derive from Theorem \ref{krull} are not canonical, hence the main results of this paper should only be interpreted as isomorphic.

The main application of Theorem \ref{krull} for our purposes is to recognize a Lie algebra as a component in a direct sum. The following two consequences will play an important role in the remainder of this paper. Since the proofs are elementary, we ommit the details.

\begin{Prop}
	\label{copies}
	Suppose that $\lie$ and $\lieh$ are finite dimensional Lie algebras over a field $E$ such that $$\underbrace{\lie \oplus \ldots \oplus \lie}_{n \text{ times}}\approx \underbrace{\lieh \oplus \ldots \oplus \lieh }_{n \text{ times}}$$ for some $n >0$, then $\lie \approx \lieh$.
\end{Prop}

\begin{Prop}
	\label{fromsum}
	Let $\lie_1, \lie_2, \lieh$ be a finite dimensional Lie algebra over a field $E$. Suppose that $\lie_1 \oplus \lieh \approx \lie_2 \oplus \lieh$, then $\lie_1 \approx \lie_2$. 
\end{Prop}

\begin{proof}
Both propositions follow directly from Theorem \ref{krull} and writing either $\lie, \lieh$ or $\lie_1, \lie_2, \lieh$ as a direct sum of indecomposable ideals.
\end{proof}

\section{Conjugate Lie algebras}
\label{sec:conj}

One of the main ingredients to realize the results in this paper is the notion of conjugate Lie algebras. In this section we introduce this construction and prove the properties needed for the remaining part of this paper. In Theorem \ref{sameunder} we show that the underlying Lie algebras of conjugate Lie algebras are isomorphic, although the Lie algebras itself might not be isomorphic, hence leading to examples for Question \ref{Q5}. The concrete examples are only given in Section \ref{sec:examples}, after we develop some more machinery in Section \ref{sec:ans} for computing the exact number of Lie algebras which have isomorphic underlying Lie algebra.

\paragraph{$\sigma$-linear maps:}

From now on, we fix a field $E$ over which we will consider vector spaces and Lie algebras. In order to introduce conjugate Lie algebras, we need to consider maps between vector spaces and Lie algebras which are not linear, but behave well under the action of a field automorphism $\sigma \in \Aut(E)$.

\begin{Def}
Let $V$ and $W$ be vector spaces over a field $E$ and $\sigma: E \to E$ a field automorphism. We say that a map $\varphi: V \to W$ is \textbf{$\sigma$-linear} if $$\varphi \left( \lambda v + \mu w \right) = \sigma(\lambda) v + \sigma(\mu) w$$ for all $v, w \in V$ and $\lambda, \mu \in E$. If $\varphi$ is moreover bijective, we call it a \textbf{$\sigma$-isomorphism} between the vector spaces $V$ and $W$.
	
In the special case where $V = \lie$ and $W = \lieh$ are Lie algebras, we say that $\varphi$ is a \textbf{$\sigma$-morphism} if it is both $\sigma$-linear and if it preserves the Lie bracket, i.e. $$\varphi \big(\left[X,Y\right] \big) = \left[\varphi \left(X\right), \varphi \left(Y \right) \right]$$ for all $X, Y \in \lie$. If $\varphi$ is bijective we call it an \textbf{$\sigma$-isomorphism} between the Lie algebras $\lie$ and $\lieh$.
\end{Def}

The following two lemmas are elementary, so we ommit the proof.

\begin{Lem}
	\label{compsigma}
	Let $E$ be any field and $U,\hspace{1mm} V,\hspace{1mm} W$ be vector spaces over the field $E$. For all field automorphisms $\sigma, \tau: E \to E$ such that $\varphi: U \to V$ is $\sigma$-linear and $\psi: V \to W$ is $\tau$-linear, the composition $$\psi \comp \varphi: U \to W$$ is $\tau \comp \sigma$-linear. If $\varphi$ and $\psi$ are moreover $\sigma$- and $\tau$-isomorphisms respectively, then $\psi \comp \varphi$ is a $\tau \comp \sigma$-isomorphism. If $U,\hspace{1mm} V$ and $W$ are Lie algebras and $\varphi, \psi$ are $\sigma$- and $\tau$-(iso)morphisms between Lie algebras respectively, then $\psi \comp \varphi$ is a $\tau \comp \sigma$-(iso)morphism between Lie algebras. If $\varphi$ is a $\sigma$-isomorphism, then $\varphi^{-1}$ is a $\sigma^{-1}$-isomorphism.
\end{Lem}

\begin{Lem}
	\label{morphoverF}
	Let $F \subseteq E$ be a field extension, $\sigma \in \Aut(E,F)$ and $V$ a vector space over $E$. Every $\sigma$-linear map $\varphi: V \to W$ induces a linear map $$\varphi_F : V_F \to W_F$$ between the underlying vector spaces over $F$. If $V = \lie$ and $W = \lieh$ are Lie algebras and $\varphi$ is a $\sigma$-morphism, then the map $\varphi_F$ is a Lie algebra morphism between the underlying Lie algebras $\varphi_F: \lie_F \to \lieh_F$. If $\varphi$ is a $\sigma$-isomorphism then $\varphi_F$ is an isomorphism between the underlying Lie algebras.
\end{Lem}

Although $\sigma$-morphisms between Lie algebras are not linear, they do share many properties with morphisms. Although this is not important for the remainder of this paper, one can show that the image $\varphi(\lie)$ of a $\sigma$-morphism $\varphi: \lie \to \lieh$  is a subalgebra of $\lieh$ and the kernel forms an ideal of $\lie$. For us, the following fact about linearly independent vectors is the most important one.
\begin{Lem}
	\label{basis}
Let $\varphi: V \to W$ be a $\sigma$-isomorphism between vector spaces. If $X_1, \ldots, X_n$ are linearly independent in $V$, then the images $\varphi(X_1), \ldots, \varphi(X_n)$ are linearly independent in $W$. 
\end{Lem}
\begin{proof} Assume that $$ \sum_{i=1}^n \lambda_i \varphi (X_i) = 0$$ for some $\lambda_i \in E$. By applying the map $\varphi ^{-1}$ we get that $ \sum_{i=1}^n \sigma^{-1}(\lambda_i) X_i = 0$. Hence we obtain that $\sigma^{-1}(\lambda_i) = 0$ since the $X_i$ are linearly independent and consequently $\lambda_i = 0$ for all $i$. We conclude that the vectors $\varphi(X_1), \ldots, \varphi(X_n)$ are linearly independent.
\end{proof} 
In particular if $\varphi: V \to W$ is a $\sigma$-isomorphism between vector spaces, then $\dim_E(V) = \dim_E(W)$. This is a property which we will use further on.

\paragraph{Construction of conjugate Lie algebras}

Given a Lie algebra $\lie$ over the field $E$ and $\sigma \in \Aut(E)$ an automorphism, we construct a Lie algebra $\lie^\sigma$ and a $\sigma$-isomorphism $\varphi^{\sigma}: \lie \to \lie^\sigma$ between Lie algebras.

First, let us describe the map between the vector spaces. Note that the automorphism $\sigma: E \to E$ induces a $\sigma$-isomorphism $\varphi^{\sigma}: E^n \to E^n$ between vector spaces by applying $\sigma$ to every entry, so $$\varphi^{\sigma}(e_1, \ldots, e_n) = (\sigma(e_1), \ldots, \sigma(e_n))$$ for all $(e_1, \ldots, e_n) \in E^n$. After fixing a vector space basis $X_1, \ldots, X_n$ for the Lie algebra $\lie$, we can identify $\lie$ with $E^n$ and hence also find a $\sigma$-isomorphism $\varphi^{\sigma}: \lie \to \lie$ on the vector space $\lie$. 

To make $\varphi^{\sigma}$ into a $\sigma$-morphism, we need to adapt the Lie bracket on $\lie$ accordingly to define a new Lie algebra $\lie^\sigma$. As a vector space, $\lie^\sigma$ is just the Lie algebra $\lie$, but the Lie bracket $[\hspace{0.3mm} \cdot , \hspace{0.3mm} \cdot]^\sigma: \lie^\sigma \times \lie^\sigma \to \lie^\sigma$ is given by

\begin{align*}
[X,Y]^\sigma = \varphi^{\sigma}\left( \left[\left(\varphi^{\sigma}\right)^{-1}(X),\left(\varphi^{\sigma}\right)^{-1}(Y) \right] \right).
\end{align*}
The map $[\hspace{0.3mm} \cdot , \hspace{0.3mm} \cdot]^\sigma: \lie^\sigma \times \lie^\sigma \to \lie^\sigma$ is linear over $E$, since
\begin{align*}
\left[\lambda X,Y\right]^\sigma &= \varphi^{\sigma}\left( \left[\left(\varphi^{\sigma}\right)^{-1}( \lambda X),\left(\varphi^{\sigma}\right)^{-1}(Y) \right] \right)\\  &= \varphi^{\sigma}\left(\left[ \sigma^{-1}(\lambda) \left(\varphi^{\sigma}\right)^{-1}(X),\left(\varphi^{\sigma}\right)^{-1}(Y) \right] \right) \\ &= \lambda \varphi^{\sigma}\left(\left[ \left(\varphi^{\sigma}\right)^{-1}(X),\left(\varphi^{\sigma}\right)^{-1}(Y) \right] \right) = \lambda [X,Y]^\sigma.
\end{align*} It is immediate that $\left[X,X\right]^\sigma = 0$ for all $X \in \lie^\sigma$. Moreover the Jacobi identity on $\lie^\sigma$ also follows from the Jacobi identity on $\lie$ since 
\begin{align*}
\left[X,\left[Y,Z\right]^\sigma\right]^\sigma = \varphi^{\sigma}\left( \left[\left(\varphi^{\sigma}\right)^{-1}(X),\left[\left(\varphi^{\sigma}\right)^{-1}(Y),\left(\varphi^{\sigma}\right)^{-1}(Z) \right]\right] \right).
\end{align*}

\noindent We conclude that $[\hspace{0.3mm} \cdot , \hspace{0.3mm} \cdot]^\sigma$ is a well-defined Lie bracket on $\lie^\sigma$ and that the map $\varphi^{\sigma}: \lie \to \lie^\sigma$ is a $\sigma$-isomorphism by construction. 

Note that the definition of the conjugate Lie algebra $\lie^\sigma$ and the $\sigma$-isomorphism $\varphi^\sigma$ depends on the initial choice of basis $X_1, \ldots, X_n$, which determines the action of $\sigma \in \Aut(E)$ on the Lie algebra $\lie$. In the next proposition we show that by changing the basis $X_1, \ldots, X_n$, we get a Lie algebra which is isomorphic.  

\begin{Prop}
	\label{uniquesigma}
Let $\varphi_1: \lie \to \lieh_1$ and $\varphi_2: \lie \to \lieh_2$ be two $\sigma$-isomorphisms between Lie algebras over the field $E$. There exists a unique isomorphism $\psi: \lieh_1 \to \lieh_2$ such that $\psi \comp \varphi_1 = \varphi_2$.
\end{Prop}

\begin{proof}
Consider the composition $\psi = \varphi_2 \comp \varphi_1^{-1}: \lieh_1 \to \lieh_2$, then Lemma \ref{compsigma} implies that this is a linear map over $E$. Since the condition also implies that it is bijective and preserves the Lie bracket, we conclude that $\psi$ is in fact an isomorphism and $\varphi_2 = \psi \comp \varphi_1$ by construction. Uniqueness follows immediately, since the maps $\varphi_1$ and $\varphi_2$ are bijections.
\end{proof}

\begin{Cor}
The $\sigma$-conjugate Lie algebra $\lie^\sigma$ and the $\sigma$-isomorphism $\varphi^\sigma: \lie \to \lie^\sigma$ are defined up to isomorphism and do not depend on the choice of basis $X_1, \ldots, X_n$ for $\lie$.
\end{Cor}

Hence from now on, we will just talk about the $\sigma$-conjugate of a Lie algebra over $E$ without mentioning the choice of basis $X_1, \ldots, X_n$. We will denote the Lie bracket on $\lie^\sigma$ as $[\hspace{0.3mm} \cdot , \hspace{0.3mm} \cdot]$ for simplicity.

\begin{Def}
Let $\lie$ be a Lie algebra over the field $E$. If $\sigma \in \Aut(E)$ is an automorphism of the field $E$, we say that $\lie^\sigma$ is the \textbf{$\sigma$-conjugate} Lie algebra of $\lie$ if there exists a $\sigma$-isomorphism $\varphi^{\sigma}: \lie \to \lie^\sigma$. We call $\lie$ and $\lieh$ conjugate over a subfield $F \subseteq E$ if there exists a $\sigma \in \Aut(E,F)$ such that $\lieh$ is the $\sigma$-conjugate of $\lie$.
\label{conjugatedef}
\end{Def}

If $\lie$ is a Lie algebra over $E$ and $X_1, \ldots, X_n$ are a basis for $\lie$, then we define the \textbf{structure constants} of $\lie$ as the elements $c_{ij}^k \in E$ such that $$[X_i,X_j] = \sum_{k=1}^n c_{ij}^k X_k.$$ Given the structure constants of the Lie algebra $\lie$ for the basis vectors $X_1, \ldots, X_n$, there is an easy way to compute the structure constants of $\lie^\sigma$ for the basis $\varphi^{\sigma}(X_1), \ldots, \varphi^{\sigma}(X_n)$. 

\begin{Ex}
	\label{conjugatestructureconstants}
	Let $X_1, \ldots, X_n$ be a basis for $\lie$ with the structure constants $c_{ij}^k$. For every $\sigma$-isomorphism $\varphi^\sigma: \lie \to \lie^\sigma$ the vectors $\varphi^\sigma(X_i)$ form a basis for $\lie^\sigma$ by Lemma \ref{basis}. If we compute the structure constants for the Lie bracket on $\lie^\sigma$, we get
	\begin{align*}
	[\varphi^\sigma(X_i),\varphi^\sigma(X_j)] &= \varphi^\sigma \left( \left[X_i,X_j \right] \right) \\ &= \varphi^{\sigma}\left( \sum_{k=1}^n c_{ij}^k X_k \right) \\ &= \sum_{k=1}^n \sigma(c_{ij}^k) \varphi^\sigma \left(X_k\right). \end{align*}
	Hence the structure constants of $\lie^\sigma$ for the basis $\varphi^{\sigma}(X_1), \ldots, \varphi^{\sigma}(X_n)$ are equal to $\sigma(c_{ij}^k)$.
\end{Ex}
Alternatively, we could have defined the $\sigma$-conjugate Lie algebra $\lie^\sigma$ from the action of $\sigma$ on the structure constants as in Example \ref{conjugatestructureconstants}. However, since the $\sigma$-isomorphism $\varphi^{\sigma}: \lie \to \lie^\sigma$ plays such an important role in the remaining part of this paper, we defined conjugate Lie algebras by the existence of the map $\varphi^\sigma$ as in Definition \ref{conjugatedef}.

\paragraph{Properties of conjugate Lie algebras:}
In this part, we introduce some properties about conjugate Lie algebras which we will use in the following parts. 

The first lemma shows that being conjugate over $F$ is a transitive relation. 
\begin{Lem}
	\label{twoconjugate}
	Let $\lie_1, \hspace{0,5mm} \lie_2$ and $\lie_3$ Lie algebra over the field $E$ and $F \subseteq E$ a subfield. Let $\sigma, \tau \in \Aut(E,F)$ be such that $\lie_1$ is $\sigma$-conjugate to $\lie_2$ and $\lie_2$ is $\tau$-conjugate to $\lie_3$, then $\lie_1$ is $\tau \comp \sigma$-conjugate to $\lie_3$. 
\end{Lem}

\begin{proof}
	By definition, there exist $\varphi^{\sigma}: \lie_1 \to \lie_2$ and $\varphi^{\tau}: \lie_2 \to \lie_3$ which are a $\sigma$- and $\tau$-isomorphism between Lie algebras, respectively. By Lemma \ref{compsigma} the composition $$\varphi^{\tau} \comp \varphi^{\sigma}: \lie_1 \to \lie_3$$ is a $\tau \comp \sigma$-isomorphism and hence the lemma holds.
\end{proof}

The next result we will use is that the conjugates of an indecomposable Lie algebra are again indecomposable. To prove this, we first study the conjugates of direct sums.

\begin{Lem}
	\label{sumsigma}
	Let $\lie$ and $\lieh$ be Lie algebras over the field $E$ and $\sigma: E \to E$ an automorphism, then there is an isomorphism between the Lie algebras $$\left( \lie \oplus \lieh \right)^\sigma \approx \lie^\sigma \oplus \lieh^\sigma.$$
\end{Lem}

\begin{proof}
	Let $\varphi^{\sigma}: \lie \to \lie^\sigma$ and $\psi^{\sigma}: \lieh \to \lieh^\sigma$ be $\sigma$-isomorphisms, which exist by the definition of $\sigma$-conjugate Lie algebra. The natural map $\varphi^{\sigma} \oplus \psi^{\sigma}: \lie \oplus \lieh \to \lie^\sigma \oplus \lieh^\sigma$ is again a $\sigma$-isomorphism, so in particular Definition \ref{conjugatedef} implies that the lemma holds.
\end{proof}

\begin{Cor}
	\label{conjindecomposable}
	Let $\lie$ and $\lieh$ be Lie algebras over the field $E$ which are conjugate, then $\lie$ is indecomposable if and only if $\lieh$ is indecomposable.
\end{Cor}

\begin{proof}
	This follows immediately from Lemma \ref{sumsigma}, since if $\lie$ is isomorphic to a non-trivial direct sum of Lie algebras, then so is $\lieh$.
\end{proof}

The next proposition on Lie algebras which are defined over a subfield $F \subseteq E$ is crucial for positive answers on Question \ref{generalQN}.
\begin{Prop}
	\label{Fform}
Let $\lie$ be a Lie algebra over $E$ which is defined over a subfield $F \subseteq E$, then every Lie algebra $\lieh$ which is conjugate over $F$ to $\lie$ is isomorphic to $\lie$. 
\end{Prop}

\begin{proof}
	Since $\lie$ is defined over $F$, there exists a basis $X_1, \ldots, X_n$ for $\lie$ such that the structure constants $c_{ij}^k \in F$. Take $\sigma \in \Aut(E,F)$ such that there exists a $\sigma$-isomorphism $\varphi^\sigma: \lie \to \lieh$. By Example \ref{conjugatestructureconstants} the structure constants for the basis $\varphi^\sigma(X_i)$ are equal to $$\sigma \left(c_{ij}^k \right) = c_{ij}^k \in F $$ and hence $\lieh$ is isomorphic to $\lie$.
\end{proof} 
\noindent Although Proposition \ref{Fform} is elementary, we can use it to show that certain Lie algebras are not defined over a field $F$ by showing that there exists $\sigma \in \Aut(E,F)$ such that $\lie^\sigma \not \approx \lie$, as we illustrate in Proposition \ref{overFprop}.

The following theorem indicates why we are interested in conjugate Lie algebras for studying underlying Lie algebras, in particular for answering Question \ref{Q5}.

\begin{Thm}
	\label{sameunder}
Let $\lie$ and $\lieh$ be Lie algebras over $E$ which are conjugate over a subfield $F \subseteq E$, then the underlying Lie algebras over $F$ are isomorphic, so $$\lie_F \approx \lieh_F.$$
\end{Thm}

\begin{proof}
By definition, there exists $\sigma \in \Aut(E,F)$ and a $\sigma$-isomorphism $\varphi^\sigma: \lie \to \lieh$. By Lemma \ref{morphoverF} $\varphi^\sigma$ induces an isomorphism between the vector spaces $\lie_F$ and $\lieh_F$. Since by definition it also preserves the Lie bracket, we conclude that $\lie_F \approx \lieh_F$.
\end{proof}

Thereom \ref{sameunder} suffices to give counterexamples for Question \ref{Q5} and we will do so in Example \ref{diffcomplexconj} in Section \ref{sec:examples}. Before giving this example, we first investigate the algebraic structure of the underlying Lie algebra $\lie_F$, allowing us to give a positive answer to Question \ref{Q5} in some specific cases and describing the number of bi-invariant complex structures on a given real Lie algebra. 

\section{Embedding the underlying Lie algebra}

\label{sec:ans}

In this section we investigate the algebraic structure of the underlying Lie algebra $\lie_F$ depending on the Lie algebra $\lie$. The main goal is to give an answer to Question \ref{generalQN} and determine which Lie algebras have isomorphic underlying Lie algebras. To do this we embed the underlying Lie algebra $\lie_F$ as a $F$-form in a direct sum of conjugate Lie algebras over $F$, see Theorem \ref{sumconjugate} underneath. By applying this result in combination with the decomposition into indecomposable ideals from Theorem \ref{krull}, we give a charactization for Lie algebras to have isomorphic underlying Lie algebra.

Let $\lie$ be a Lie algebra over any field $E$. From now on, we restrict our attention to Galois extensions $F \subseteq E$ of finite degree. The results could also be formulated for general seperable extensions of finite degree, but this would make the statements more technical. This seems unneccesary since every finite degree field extension in characteristic $0$ lies in a Galois extension of finite degree, hence offering a wide range of applications in Section \ref{sec:examples}.

\begin{Thm}
	\label{sumconjugate}
	Let $F \subseteq E$ be a Galois extension and let $\lie$ be a Lie algebra over the field $E$. Let $\lie_F$ be the underlying Lie algebra over $F$, then there is an isomorphism $$ \lie_F \otimes_F E \approx \bigoplus_{\sigma \in \Gal(E,F)} \lie^\sigma $$ where $\lie^\sigma$ denotes the $\sigma$-conjugate of the Lie algebra $\lie$ as in Definition \ref{conjugatedef}. 
\end{Thm}

\noindent The techniques in this proof are similar to the ones in \cite[Lemma 2.1]{dere13-1} for constructing rational forms in Lie algebras. 

\begin{proof}
Consider the maps $\varphi^\sigma: \lie \to \lie^\sigma$ for every $\sigma \in \Gal(E,F)$. Define the natural map \begin{align*}
\varphi: \lie &\to \bigoplus_{\sigma \in \Gal(E,F)} \lie^\sigma\\& X \mapsto \Big( \varphi^\sigma(X) \Big)_{\sigma \in \Gal(E,F)} \end{align*}
 which is given by $\varphi^\sigma$ in every component. The induced injective map $$\varphi_F: \lie_F \to \left( \bigoplus_{\sigma \in \Gal(E,F)} \lie^\sigma \right)_F$$ is linear over $F$ and we will show that the image of $\varphi_F$ is an $F$-form of $\left( \oplus_{\sigma \in \Gal(E,F)} \lie^\sigma \right)$ as a Lie algebra over $E$. Note that $$\dim_F(\varphi_F(\lie_F)) = \dim_F(\lie_F) = [E:F] \dim_E(\lie) = \dim_E \left( \oplus_{\sigma \in \Gal(E,F)} \lie^\sigma \right),$$ so it suffices to show that the image under $\varphi_F$ of linearly independent vectors in $\lie_F$ over $F$ are linearly independent in $\bigoplus_{\sigma \in \Gal(E,F)} \lie^\sigma$ over $E$.
 
 To prove this, assume for a contradiction that this is not the case. Take $X_1, \ldots, X_k \in \lie_F$ which are linearly independent over $F$ and such that there exists $\lambda_1, \ldots, \lambda_k \in E$ with at least one $\lambda_i$ non-zero and 
 \begin{align}
 \label{lincom} 0 = \sum_{i=1}^k \lambda_i \varphi(X_i) = \left( \sum_{i=1}^k \lambda_i \varphi^\sigma(X_i) \right)_{\sigma \in \Gal(E,F)} \in \bigoplus_{\sigma \in \Gal(E,F)}\lie^\sigma.
 \end{align} Without loss of generality we can assume that $k$ is the minimal number for which such an example exists and hence it holds that all $\lambda_i \neq 0$. By multiplying by $\lambda^{-1}_1$ if necessary we can assume that $\lambda_1 = 1$. The minimality of $k$ moreover implies that the $\lambda_i$ are unique under the latter assumption. 
 
 Equation (\ref{lincom}) is equivalent, by considering every component seperately, to 
 \begin{align} \label{seplincom} \sum_{i=1}^k \lambda_i \varphi^\sigma(X_i) = 0\end{align} for every $\sigma \in \Gal(E,F)$. We now claim that for every $\sigma, \tau \in \Gal(E,F)$ the equality $$\sum_{i=1}^k \tau (\lambda_i) \varphi^\sigma(X_i) = 0$$ holds. To show this, we apply the map $\varphi^{\sigma} \comp \left(\varphi^{\tau^{-1} \comp \sigma}\right)^{-1}$ to Equation (\ref{seplincom}) for the automorphism $\tau^{-1} \comp \sigma$ and get 
 
 \begin{align*}
 0 &= \varphi^{\sigma} \comp \left(\varphi^{\tau^{-1} \comp \sigma}\right)^{-1} \left( \sum_{i=1}^k \lambda_i \varphi^{\tau^{-1} \comp \sigma}(X_i) \right) \\
 &= \sum_{i=1}^k \varphi^\sigma\left( \sigma^{-1} \left(\tau \left(\lambda_i \right) \right) X_i \right) \\  &= \sum_{i=1}^k \tau(\lambda_i) \varphi^\sigma(X_i) .
 \end{align*} 
 
 Fixing $\tau \in \Gal(E,F)$ and using that the claim holds for all $\sigma \in \Gal(E,F)$, we conclude that $\tau(\lambda_i) = \lambda_i$ by the fact that $\tau(\lambda_1) = \tau(1) = 1$ and the uniqueness of the $\lambda_i$. Because this holds for every $\tau \in \Gal(E,F)$, we conclude that $\lambda_i \in F$, leading to a contradiction to the fact that $X_1, \ldots, X_k$ are linearly dependent over $F$.
\end{proof}

As an application of this result, we get a positive answer on Question \ref{Q5} under some specific assumptions for the Lie algebras $\lie$ and $\lieh$.

\begin{Cor}
	\label{partialanswer}
Let $F \subseteq E$ be a Galois extension and $\lie$ a Lie algebra over $E$ which is defined over the subfield $F$. If $\lieh$ is another Lie algebra over $E$ such that $\lie_F \approx \lieh_F$, then the original Lie algebras $\lie \approx \lieh$ are isomorphic as well if either
\begin{enumerate}
	\item the Lie algebra $\lieh$ is defined over $F$, or
	\item the Lie algebra $\lie$ is indecomposable.
	\end{enumerate}
\end{Cor}

\begin{proof}
	Since $\lie$ is defined over $F$, every conjugate Lie algebra $\lie^\sigma$ is isomorphic to $\lie$, see Proposition \ref{Fform}. By applying Theorem \ref{sumconjugate}, the Lie algebra $\lie_F$ is an $F$-form of $\underbrace{\lie \oplus \ldots \oplus \lie}_{[E:F] \text{ times}}$. We assume that $\lieh$ is a Lie algebra with $\lie_F \approx \lieh_F$. 
	
	If $\lieh$ is defined over $F$, then Theorem \ref{sumconjugate} implies in exactly the same way that $\lieh_F$ is a $F$-form of $\underbrace{\lieh \oplus \ldots \oplus \lieh}_{[E:F] \text{ times}}$. Consequently, by extending the field to $E$, Proposition \ref{copies} implies that $\lie \approx \lieh$.  
	
	\medskip
	
	\noindent If $\lie$ is indecomposable, we find that $$\underbrace{\lie \oplus \ldots \oplus \lie}_{[E:F] \text{ times}} \approx \bigoplus_{\sigma \in \Gal(E,F)} \lieh^\sigma.$$ Since $\lie$ is indecomposable and the number of components correspond, we find that the Lie algebras $\lieh^\sigma$ are also indecomposable. In particular, since $\lieh$ is a indecomposable ideal of the right hand side, it is isomorphic to an indecomposable ideal on the left hand side, so isomorphic to $\lie$.
\end{proof}
\noindent In particular, in the second case of Corollary \ref{partialanswer} we conclude that $\lieh$ is defined over $F$ as well. As we will demonstrate in Example \ref{definedoverF}, the conditions on $\lie$ or $\lieh$ are necessary for this corollary.  

More generally, given two Lie algebras $\lie$ and $\lieh$ over $E$, we will fully characterize when the underlying Lie algebras $\lie_F \approx \lieh_F$ are isomorphic. Observe that $\left( \lie \oplus \lieh \right)_F = \lie_F \oplus \lieh_F$ for all Lie algebras $\lie$ and $\lieh$ over $E$. Together with Theorem \ref{sameunder} this implies that for all $k \in \N$, if $\sigma_i \in \Gal(E,F)$ and $\lie_i$ are Lie algebras over $E$ for $1 \leq i \leq k$, we find that $$\left(\bigoplus_{i=1}^k \lie_i \right)_F \approx \left( \bigoplus_{i=1}^k \lie_i^{\sigma_i} \right)_F.$$ Our next theorem shows that the converse holds as well if the $\lie_i$ are indecomposable. This is hence the best possible result in this direction.

\begin{Thm}
	\label{classification}
Let $F \subseteq E$ be a Galois extension and $\lie$ and $\lieh$ be Lie algebras over $E$ such that $\lie_F \approx \lieh_F$. If we decompose $$\lie \approx \bigoplus_{i=1}^k \lie_i$$ into indecomposable ideals, then there exists $\sigma_i \in \Gal(E,F)$ such that $$\lieh \approx \bigoplus_{i=1}^k \lie_i^{\sigma_i}.$$
\end{Thm}

\begin{proof}
Write $\lie = \bigoplus_{i=1}^k \lie_i$ and $\lieh = \bigoplus_{j=1}^l \lieh_j$ with the ideals $\lie_i, \hspace{0.5mm} \lieh_j$ indecomposable. Since $\lie_F \approx \lieh_F$ and hence also $\lie_F \otimes_F E \approx \lieh_F\otimes_F E$, applying Theorem \ref{sumconjugate} gives us 
\begin{align} \label{induction1} \bigoplus_{i=1}^k \left( \bigoplus_{\sigma \in \Gal(E,F)} \lie_i^\sigma \right) \approx \bigoplus_{j=1}^l \left( \bigoplus_{\sigma \in \Gal(E,F)} \lieh_j^\sigma \right).
\end{align}
Since $\lie_i$ and $\lieh_j$ are indecomposable, we have that for every $\sigma \in \Gal(E,F)$, the Lie algebras $\lie_i^\sigma$ and $\lieh_j^\sigma$ are indecomposable by Corollary \ref{conjindecomposable}. By comparing the number of indecomposable ideals on the left hand side and right hand side of the isomorphism in (\ref{induction1}) we get that $k = l$. We prove the theorem by induction on the number $k$ of indecomposable ideals of $\lie$.

First assume that $k= 1$ or thus that $\lie$ is indecomposable. Since $\lieh$ is an indecomposable component of the right hand side in (\ref{induction1}), Theorem \ref{krull} implies that $\lieh \approx \lie^\sigma$ for some $\sigma \in \Gal(E,F)$, which proves the theorem in this case.

For the induction step, we take $k > 1$. Since $\lieh_1$ is an indecomposable component of the right hand side of (\ref{induction1}), there exists a $1 \leq j \leq k$ and a $\sigma_1 \in \Gal(E,F)$ such that $\lieh_1 \approx \lie_j^{\sigma_1}$. Up to renumbering the ideals of $\lie$, we can assume that $j = 1$. For every $\sigma \in \Gal(E,F)$, we find that $\lieh_1^\sigma \approx \left( \lie_1^{\sigma_1} \right)^\sigma \approx \lie_1^{\sigma \comp \sigma_1}$ by Lemma \ref{twoconjugate}. Hence $$\bigoplus_{\sigma \in \Gal(E,F)} \lieh_1^\sigma \approx \bigoplus_{\sigma \in \Gal(E,F)} \lie_1^{\sigma \comp \sigma_1} \approx \bigoplus_{\sigma \in \Gal(E,F)} \lie_1^\sigma$$ and by applying Proposition \ref{fromsum} we thus get that 
\begin{align} 
\label{induction2}
\bigoplus_{i=2}^k \left( \bigoplus_{\sigma \in \Gal(E,F)} \lie_i^\sigma \right) \approx \bigoplus_{j=2}^l \left( \bigoplus_{\sigma \in \Gal(E,F)} \lieh_j^\sigma \right).
\end{align}
By applying the induction hypothesis to the isomorphism in (\ref{induction2}), we get the statement of the theorem.
\end{proof}
Note that Theorem \ref{classification} implies that if two Lie algebras have isomorphic underlying Lie algebras, they have the same number of indecomposable ideals in a decomposition. As another consequence, we conclude that there are only a finite number of Lie algebras over $E$ which can have isomorphic underlying Lie algebras.

\begin{Cor}
	\label{finitenumber}
Let $F \subseteq E$ be a Galois extension. Given a Lie algebra $\lieh$ over the field $F$, there exist only finitely many (possibly none) Lie algebras $\lie$ over the field $E$ such that $$\lie_F \approx \lieh.$$
\end{Cor}

\begin{proof}
This follows from Theorem \ref{classification}, since there are only a finite number of possibilities for the $\sigma_i \in \Gal(E,F)$ and a finite number of indecomposable ideals.
\end{proof}

\noindent So in particular, a real Lie algebra admits only a finite number of bi-invariant complex structures up to isomorphism. We will give a method to compute the exact number in the following section.

\section{Applications}

\label{sec:examples}

This section demonstrates some applications of the main results. We focus on $\R \subseteq \C$ and Galois extensions $\Q \subseteq E$ because of the importance for applications, but the methods work for general Galois extensions.

Recall from Example \ref{CoverR} that we denote the non-trivial element in $\Gal(\C,\R)$, namely the complex conjugation map, as $\overline{ \phantom{z}}: \C \to \C$. In the first subsection we construct negative examples for Question \ref{Q5}, where the strategy is to apply Theorem \ref{sameunder} on a complex $2$-step Lie algebra $\lie$ such that $\lie \not \approx \overline{\lie}$, where $\overline{\lie}$ is the complex conjugate Lie algebra of $\lie$. To show that the complex conjugate Lie algebra is not isomorphic to the original one, we use invariants constructed from the Pfaffian form. The second subsection shows how to apply our main results as a tool to determine which Lie algebras in dimension $\leq 4$ are defined over $\Q$ and $\R$. 

\subsection{Bi-invariant complex structures on $2$-step nilpotent Lie algebras} A necessary condition for complex Lie algebras $\lie$ with $\lie \not \approx \overline{\lie}$ is that $\lie$ has no real forms. It is easy to construct such examples, for example as a semi-direct product $\C^2 \rtimes \C$, see \cite{same69-1}, but unfortunately this is harder to achieve for nilpotent Lie algebras. We start by recalling the approach from \cite{dlv12-1} for constructing such examples. 

The essence of this method is to associate to every $2$-step nilpotent Lie algebra $\lien$ a certain polynomial called the Pfaffian form. Using the properties of the Pfaffian form we construct an invariant $c_{\lien} \in \C$ such that for the complex conjugate Lie algebra $\overline{ \lien }$ the invariant is equal to $c_{\overline{\lien}} = \overline{c_{\lien}}$. Our example will be a Lie algebra for which $c_{\lien} \in \C \setminus \R$. 

Let $\lien$ be a $2$-step nilpotent Lie algebra of dimension $n$ over the field $E$. Write $q = \dim_E \left( [\lien,\lien] \right)$ for the dimension of the commutator subalgebra and $p = n-q$, then we say in short that the Lie algebra $\lien$ is of type $(p,q)$. If $p$ is even, then we can associate to $\lien$, with a given basis, a homogeneous polynomial $f_\lien(x_1,\ldots, x_q)$ of degree $\frac{p}{2}$ in $q$ variables with coefficients in $E$ which we call the \textbf{Pfaffian form} of $\lien$. Hereinafter we will denote the vector space of homogeneous polynomials of degree $\frac{p}{2}$ in $q$ variables as $P_{q,\frac{p}{2}}$, so in particular $f_\lien \in P_{q,\frac{p}{2}}$. 

The Pfaffian form depends on the choice of basis for $\lien$, hence it is not unique, but if two Lie algebras $\lien_1, \lien_2$ are isomorphic, then the Pfaffian forms are \textbf{projectively equivalent}, meaning that there exists a matrix $A \in \GL(q,E)$ and a constant $k \in E^\ast$ such that $$f_{\lien_1}(x_1, \ldots, x_q) = k f_{\lien_2}(A(x_1,\ldots, x_q)).$$ For more details about the construction of the Pfaffian form and the proof of this statement we refer to \cite{laur08-1}. 

We now construct invariants for a complex $2$-step nilpotent Lie algebra $\lien$ from the Pfaffian form $f_\lien$. Suppose that $$S, T: P_{q,\frac{p}{2}} \to \C$$ are homogeneous polynomials of the same degree which are $SL(2,\C)$-invariant. It is an exercise to check that in this case if $f_1, f_2 \in P_{q,\frac{p}{2}}$ are projectively equivalent then $$\frac{S(f_1)}{T(f_1)} = \frac{S(f_2)}{T(f_2)}.$$ Hence for a Lie algebra $\lien$ of type $(p,q)$, we can define an invariant $  \frac{S(f_\lien)}{T(f_\lien)} \in \C$ which in particular does not depend on the choice of a basis.

We demonstrate this technique for the specific example of type $(8,2)$, where we use two different homogeneous polynomials $S, T : P_{2,4} \to \C$. The first one of degree $2$ is defined as $$S(a x^4 + b x^3 y + c x^2 y^2 + d x y^3 + e y^4) = ae -4bd + 3c^2$$ whereas the second one of degree $3$ is given by $$T(a x^4 + b x^3 y + c x^2 y^2 + d x y^3 + e y^4) = a c e - a d^2 + 2 b c d - b^2 e - c^3.$$ On \cite[Page 150]{dolg03-1} it is shown that these are indeed invariant under $\SL(2,\C)$, leading to the invariant $$c_\lien = \frac{\left( S(f_\lien) \right)^3}{ \left(T(f_\lien) \right)^2}$$ for the Lie algebras $\lien$ of type $(8,2)$. Note that the coefficients of the polynomials $S$ and $T$ are real numbers, even integers.

Now we are ready to show that the following example of \cite[Example 4.4.]{dlv12-1} has a complex conjugate which is not isomorphic. This example will be the building block for many other types of examples.
\begin{Ex}
\label{diffcomplexconj}
Let $\lie_\lambda$ be the Lie algebra of dimension $10$ with basis $X_1, \ldots, X_8, Z_1, Z_2$ and Lie bracket $[\cdot,\cdot ]$ defined as
\begin{align*}
[X_1,X_5] &= Z_1 & [X_2,X_6] &= Z_1 & [X_3,X_7] &= Z_1\\
[X_4,X_8] &= Z_1 & [X_2,X_5] &= Z_2 & [X_3,X_6] &= Z_2\\
[X_4,X_7] &= Z_2 & [X_1,X_8] &= - Z_2 & [X_2,X_7] &= - \lambda Z_2.
\end{align*}
Using the definition of the Pfaffian form in \cite{laur08-1}, a computation shows that $$f_\lambda(x,y) :=f_{\lie_{\lambda}} (x,y) =  x^4 + \lambda x^2 y^2 + y^4.$$ Note that the Lie algebra $\lie_\lambda$ is indecomposable. By Example \ref{conjugatestructureconstants} the complex conjugate Lie algebra of $\lie_\lambda$ satisfies $\overline{ \lie_\lambda} \approx \lie_{\overline{ \lambda}}$.
\end{Ex}

The Lie algebra of Example \ref{diffcomplexconj} was already used in \cite{dlv12-1} for giving complex Lie algebras which are not defined over $\R$. We apply the same techniques to find Lie algebras for which the complex conjugate Lie algebra is not isomorphic to the original one.

\begin{Prop}
	\label{onlydiffconj}
There exists $\lambda \in \C$ such that the complex Lie algebras $\lie_\lambda$ and $\lie_{\overline{ \lambda}}$ of Example \ref{diffcomplexconj} are not isomorphic.
\end{Prop}

\begin{proof}
Take the polynomials $S, T$ as introduced just above Example \ref{diffcomplexconj}. By evaluating in the polynomial $f_\lambda$ we get $S(f_\lambda) = 3\lambda^2 + 1$ and $T(f_\lambda) = \lambda - \lambda^3$ and thus the invariant introduced above Example \ref{diffcomplexconj} is equal to $$c(\lambda) := c_{\lie_{\lambda}} = \frac{S^3(f_\lambda)}{T^2(f_\lambda)} = \frac{\left(3 \lambda^2 + 1 \right)^3}{\left( \lambda - \lambda^3 \right)^2}$$ Now pick any $\lambda \in \C$ such that $c(\lambda) \notin \R$, then $c\big( \overline{ \lambda} \big) = \overline{c \left( \lambda \right) } \neq c \left( \lambda \right)$ and hence $\lie_{\lambda} \not \approx \lie_{\overline{\lambda}}$ which proves the proposition.
\end{proof}

\noindent There exist $\lambda \in \C \setminus \R$ such that $c(\lambda) \in \R$, for example $c(i) = 2$, which make the invariant $c(\lambda)$ inapplicable for investigating whether or not $\lie_{i}$ and $\lie_{-i}$ are isomorphic. Note that Proposition \ref{onlydiffconj} is true for every Lie algebra $\lien$ of type $(8,2)$ for which the quotient $\frac{S^3(f_\lien)}{T^2(f_\lien)}$ does not lie in $\R$, in particular also for Lie algebras as in \cite[Example 4.5.]{dlv12-1}. 

In Corollary \ref{partialanswer} we gave some conditions under which we recover the original Lie algebra from the underlying Lie algebra. Starting from Proposition \ref{onlydiffconj} we construct $2$-step nilpotent examples illustrating the necessity of these conditions. 
\begin{Ex}
	\label{definedoverF}
Take $\lambda \in \C$ such that $\lie_{\lambda} \not \approx \lie_{\overline{ \lambda}}$ with $\lie_{\lambda}$ the Lie algebra as in Example \ref{diffcomplexconj}. Define the Lie algebra $\lie = \lie_{\lambda} \oplus \lie_{\overline{\lambda}}$, which is defined over $\R$ since it contains $\left(\lie_{\lambda}\right)_\R$ as a real form by Theorem \ref{sumconjugate}. Note that the Lie algebra $\lieh = \lie_{\lambda} \oplus \lie_{\lambda}$ satisfies $\lie_{\R} \approx \lieh_\R$ by Theorem \ref{sameunder}, but that the Lie algebras $\lie$ and $\lieh$ are not isomorpic since $\lie$ has $\lie_{\overline{\lambda}}$ as an indecomposable component and $\lieh$ on the contrary does not. Comparing to Corollary \ref{partialanswer}, we point out that the Lie algebra $\lieh$ is not defined over $\R$ and is not indecomposable.
\end{Ex}

Using the indecomposable Lie algebra of Example \ref{diffcomplexconj} we are able to construct examples with exactly $n$ different bi-invariant complex structures.
\begin{Ex}
	\label{nintot}
Let $\lie_\lambda$ be the Lie algebra as in Example \ref{diffcomplexconj} for any $\lambda \in \C$ as in Proposition \ref{onlydiffconj}. Consider the complex Lie algebra $$\lie = \underbrace{\lie_\lambda \oplus \ldots \oplus \lie_\lambda}_{k \textit{ components}},$$ then we show that the underlying Lie algebra $\lie_\R$ has exactly $k+1$ different bi-invariant complex structures. 
	 
First of all, consider the Lie algebras	$$ \lieh_j =  \underbrace{\lie_{\lambda} \oplus  \ldots \oplus \lie_{\lambda} }_{j \textit{ components}} \oplus \underbrace{\overline{\lie_{\lambda}} \oplus  \ldots \oplus \overline{\lie_{\lambda}} }_{k - j \textit{ components}}$$ for $0 \leq j \leq k$, so in particular $\lieh_k = \lie$. The Lie algebras $\lieh_j$ are all non-isomorphic because of Theorem \ref{krull} and Proposition \ref{onlydiffconj}. Moreover, Theorem \ref{sameunder} implies that the underlying real form $(\lieh_j)_\R$ is isomorphic to $\lie_{\R}$, leading to at least $k+1$ different bi-invariant complex structures on $\lie_{\R}$.
	
Next we show that if $\lieh$ is a complex Lie algebra such that $\lieh_{\R} \approx \lie_{\R}$, then $\lieh \approx \lieh_j$ for some $0 \leq j \leq k$. In fact, this follows directly from Theorem \ref{classification} and using that every conjugate of $\lie_{\lambda}$ over $\R$ is either isomorphic to $\lie_{\lambda}$ or $\overline{ \lie_\lambda}$. This shows that indeed $\lie_{\R}$ has exactly $k+1$ different bi-invariant complex structures.
\end{Ex}
Example \ref{definedoverF} works for every indecomposable complex Lie algebra $\lie$ such that $\overline{\lie} \not \approx \lie$. As an immediate consequence of this example we have the following statement.

\begin{Cor}
For every $n \in \N$, there exists a real $2$-step Lie algebra $\lie_\R$ such that $\lie_\R$ has exactly $n$ different bi-invariant complex structures.
\end{Cor} 

More generally, the method of Example \ref{nintot} leads to the following result which counts the number of bi-invariant complex structures of an underlying real Lie algebra, giving an explicit form of Corollary \ref{finitenumber}.

\begin{Thm}
Let $\lie$ be a complex Lie algebra with decomposition $\lie = \bigoplus_{i=1}^l \lie_i$ into ideals such that $$\lie_i = \bigoplus_{j=1}^{k_i} \lie_{ij}$$ is a direct sum of indecomposable ideals with either $\lie_{i1} \approx \lie_{ij}$ or $\lie_{i1} \approx \overline{\lie_{ij}}$ for all $1 \leq j \leq k_i$. Moreover we assume that if $\lie_{i1} \approx \lie_{i^\prime 1}$ or $\lie_{i1} \approx \overline{\lie_{i^\prime 1}}$, then $i = i^\prime$ and that there exists $1 \leq l_0 \leq l$ such that $\lie_{i1} \not \approx \overline{\lie_{i1}}$ if and only $i \leq l_0$.  There exists up to isomorphism exactly $$\prod_{i=1}^{l_0} \left(k_i + 1 \right)$$ complex Lie algebras $\lieh$ such that $\lie_\R \approx \lieh_{\R}$.
\end{Thm}

\noindent 
The proof is exactly the same as the argument in Example \ref{nintot}, by on the one hand constructing Lie algebras to show the lower bound and using Theorem \ref{classification} to show that we achieved all possible Lie algebras. We leave the details to the reader to check.

\subsection{Real and rational forms in low dimensions}
\label{sec:ex2}

Another application of Theorem \ref{sameunder} in combination with Proposition \ref{Fform} is a method to verify which Lie algebras over $E$ are defined over the subfield $F \subseteq E$. We demonstrate this by describing the complex Lie algebras up to dimension $4$ which are definied over $\R$ and over $\Q$. The complex Lie algebras up to dimension $4$ have been classified and a list can be found for example in \cite{bs99-1}. We will use the notations from the latter paper, in particular the Lie algebras as defined in Lemma 2 and 3.

The importance of determining whether a Lie algebra is defined over $\Q$ or over $\R$ lies in geometric applications. If a Lie algebra is defined over $\R$, then the Lie algebra is the complexification of a Lie algebra corresponding to a Lie group. Furthermore, in the nilpotent case Mal'cev showed in \cite{malc49-2} that a $1$-connected nilpotent Lie group admits a cocompact lattice if and only if the corresponding Lie algebra is defined over $\Q$. Hence the fields $\Q$ and $\R$ are the most important ones for studying lattices in nilpotent Lie groups.

Note that most Lie algebras on \cite[Page 4]{bs99-1} are already given by a basis with rational structure constants, so there is only a limited number of examples we have to check. More specifically, the only Lie algebras which are a priori not defined over $\Q$ are the ones depending on a parameter $\lambda, \alpha, \beta \in \C$, namely $\lier_{3,\lambda}(\C), \lier_{3,\lambda}(\C) \oplus \C, \lie_1(\alpha), \lie_2(\alpha,\beta), \lie_3(\alpha), \lie_8(\alpha)$. We will first deal with the last four types and treat with the first two afterwards.

If a complex Lie algebra $\lie$ is defined over $\R$, then the complex conjugate $\overline{\lie}$ must be isomorphic to $\lie$ by Proposition \ref{overF}. Consider for every $\alpha \in \C^\ast$ the Lie algebra $\lie_1(\alpha)$ of dimension $4$ for which the Lie bracket of the basis vectors $X_1, X_2, X_3, X_4$ is determined by 
\begin{align*}
[X_1, X_2]  & = X_2\\
[X_1,X_3] &= X_3\\
[X_1,X_4] &= \alpha X_4.
\end{align*}
Note that $\lie_1(\alpha) \approx \lie_1(\beta)$ if and only if $\alpha = \beta$. Since $\overline{ \lie_1(\alpha) } \approx \lie_1(\overline{ \alpha})$ by Example \ref{conjugatestructureconstants}, we see that $\lie_1(\alpha)$ is defined over $\R$ if and only if $\alpha \in \R^\ast$. A similar argument deals with the Lie algebras $\lie_2(\alpha, \beta)$, $\lie_3(\alpha)$ and $\lie_8(\alpha)$ of \cite{bs99-1}, showing that these Lie algebras are defined over $\R$ if and only if the defining parameters are real. 

To check for which parameters $\alpha \in \C^\ast$ the Lie algebra $\lie_1(\alpha)$ is defined over the rationals, we can make a similar analysis, by replacing the field extension $\R \subseteq \C$ by the extension $\Q \subseteq \Q(\alpha) \subseteq \C$. If we assume that $\alpha \notin \Q$, then either $\alpha$ is algebraic over $\Q$, meaning that $[\Q(\alpha): \Q] < \infty$, or $\alpha$ is transcendental over $\Q$, meaning that $[\Q(\alpha):\Q] = \infty$. In the first case there exists a Galois extension $E$ of $\Q$ with $\alpha \in E$, in the second case we take $E = \Q(\alpha)$. both leading to an automorphism $\sigma \in \Aut(E)$ with $\sigma(\alpha) \neq \alpha$. By taking an extension $\tilde{\sigma}: \C \to \C$ of $\sigma$ to $\C$, we get that $\lie_1(\alpha)^{\tilde{\sigma}} \approx \lie_1(\tilde{\sigma}(\alpha)) = \lie_1(\sigma(\alpha))$ is not isomorphic to $\lie_1(\alpha)$. This shows that the Lie algebra $\lie_1(\alpha)$ and similarly the Lie algebras $\lie_2(\alpha, \beta)$, $\lie_3(\alpha)$ and $\lie_8(\alpha)$ of \cite{bs99-1} are defined over $\Q$ if and only if the defining parameters are rational.

The final cases which remain are the $3$-dimensional Lie algebra $\lier_{3,\lambda}(\C)$ and the $4$-dimensional Lie algebra $\lier_{3,\lambda}(\C) \oplus \C$ for $\lambda \in \C^\ast$. The Lie algebra $\lier_{3,\lambda}(\C)$ has basis $X_1, X_2, X_3$ and Lie bracket given by
\begin{align*}
[X_1, X_2] &= X_2\\
[X_1,X_3] &= \lambda X_3.
\end{align*}
The Lie algebras $\lier_{3,\lambda}(\C) \approx \lier_{3,\mu}(\C)$ are isomorphic if and only if $\lambda = \mu^{\pm 1}$. By Proposition \ref{fromsum}, this also implies that the Lie algebras $\lier_{3,\lambda}(\C) \oplus \C\approx \lier_{3,\mu}(\C) \oplus \C$ are isomorphic if and only if $\lambda = \mu^{\pm 1}$. We answer the question over which fields $\Q \subseteq F \subseteq \C$ of characteristic $0$ these Lie algebras are defined in full generality.

\begin{Prop}
	\label{overFprop}
	The complex Lie algebra $\lier_{3,\lambda}(\C)$ or $\lier_{3,\lambda}(\C) \oplus \C$ with $\lambda \in \C^\ast$ is defined over $F$ if and only if 
	\begin{align} \label{overF} \lambda^2 + a \lambda + 1 = 0\end{align} for some $a \in F$.
\end{Prop} 

\begin{proof}
	We only give the proof for $\lier_{3,\lambda}(\C)$ since the other case is identical. 
	
	To show that Equation (\ref{overF}) is a necessary condition, we assume that $\lier_{3,\lambda}(\C)$ is defined over $F$ and we consider the field $E = F(\lambda)$. First we show that $\lambda$ has to be algebraic over $F$. If not, then there would exist an automorphism $\sigma \in \Aut(E,F)$ such that $\sigma(\lambda) = 2 \lambda$. By extending $\sigma$ to an automorphism $\tilde{\sigma}: \C \to \C$, Proposition \ref{Fform} and Example \ref{conjugatestructureconstants} imply that $\left( \lier_{3,\lambda}(\C) \right)^{\tilde{\sigma}} \approx \lier_{3,\sigma(\lambda)} (\C)= \lier_{3,2\lambda}(\C)$ and $\lambda \neq 2 \lambda \neq \lambda^{-1}$, we get a contradiction, so $\lambda$ must be algebraic over $F$.
	
	Similarly as the previous argument, it follows that for every conjugate $\mu$ of $\lambda$ over $F$ that $\lier_{3,\lambda}(\C) \approx \lier_{3,\mu}(\C)$ by Proposition \ref{overF}. Since this is only possible for $\lambda = \mu^{\pm 1}$, we conclude that $\lambda$ can have at most two conjugates. If $\lambda$ has only one conjugate, then $\lambda \in F$ and hence Equation (\ref{overF}) is satisfied for $a = - \frac{1 + \lambda^2}{\lambda} \in F$. If $\lambda$ has two conjugates $\lambda$ and $\lambda^{-1}$, then the minimal polynomial of $\lambda$ over $F$ has degree $2$ and constant term $\lambda \lambda^{-1} = 1$, so Equation (\ref{overF}) is satisfied for the minimal polynomial over $F$.
	
	Next, we show that Equation (\ref{overF}) is a sufficient condition. Assume that $\lambda$ satisfies Equation (\ref{overF}) for some $a \in F$. If $\lambda \in F$, then of course the Lie algebra $\lier_{3,\lambda}(\C)$ is defined over $F$, so we can assume that $\lambda \notin F$. Note that $a = 2$ implies that $\lambda = - 1 \in F$ and in particular we get $a \neq 2$. Write $E = F(\lambda) \supseteq F$, which is a Galois extension of degree $2$ and take $1 \neq \sigma \in \Gal(E,F)$ the non-trivial element of the Galois group. Note that $\sigma(\lambda) = \lambda^{-1}$, so in particular $\lier_{3,\lambda}(E) \approx \lier_{3,\sigma(\lambda)}(E) \approx \left( \lier_{3,\lambda}(E)\right)^\sigma$, which is a necessary condition for being defined over $F$. 
	
	To show that $\lier_{3,\lambda}(\C)$ is defined over $F$, we give an explicit Lie algebra $\lie$ over $\C$ with structure constants in $F$ which is isomorphic to $\lier_{3,\lambda}(\C)$. The Lie algebra $\lie$ is given by the basis $Y_1, Y_2, Y_3$ and Lie bracket 
	\begin{align*}
	[Y_1,Y_2] &= Y_3,\\
	[Y_1,Y_3] &= \left( a - 2 \right) Y_2 + \left(2 - a \right) Y_3, \\
	[Y_2,Y_3] &= 0.
	\end{align*}
    Consider the basis $X_1, X_2, X_3$ given by 
	\begin{align*} 
	X_1 &= \frac{1}{\lambda + 1} Y_1 = - \frac{\sigma(\lambda) + 1}{a - 2} Y_1 \\
	X_2 &= - (\sigma(\lambda) + 1 ) Y_2 +  Y_3 \\
	X_3 &= - (\lambda + 1) Y_2 + Y_3.
	\end{align*} A computation, where we use $\lambda + \sigma(\lambda) = -a$, shows that 
	\begin{align*}
	[X_1,X_2] &= \frac{1}{\lambda + 1} \left( (a-2) Y_2 + (2-a) Y_3 - (\sigma(\lambda) + 1) Y_3 \right) \\
	& = -(\sigma(\lambda) + 1) Y_2 + \frac{1 - a - \sigma(\lambda)}{\lambda + 1}Y_3 = X_2 \\
	[X_1,X_3] &= \frac{1}{\lambda + 1} \left( (a-2) Y_2 + (2-a) Y_3 - (\lambda + 1) Y_3 \right) \\ &= -(\sigma(\lambda) + 1) Y_2 + \frac{1 - a - \lambda }{\lambda+1} Y_3 = \sigma(\lambda) X_3 \\
		[X_2,X_3] &= 0
	\end{align*} 
	and thus $\lie \approx \lier_{3,\sigma(\lambda)}(\C) \approx \lier_{3,\lambda}(\C)$. This ends the proof of the proposition.
\end{proof}

\noindent For example the Lie algebra $\lier_{3,i}(\C)$ is defined over $\R$, although the defining parameter does not lie in $\R$. We summarize Proposition \ref{overFprop} for the special case $F = \R$ as follows.

\begin{Prop}
	\label{conjfor4}
Let $\lie$ be a complex Lie algebra of dimension $\leq 4$, then $\lie$ is defined over $\R$ if and only if $\lie \approx \overline{\lie}$.
\end{Prop}

\noindent We conjecture that this is the case for all complex Lie algebras, see Section \ref{openquestions}. The methods of this section can be applied for every field $E$ and any classification of Lie algebras, for example also for nilpotent Lie algebras up to dimension $7$. 

\section{Open questions}
\label{openquestions}

The results of this paper open new directions for further research on Lie algebras over different fields. This sections gathers these open questions and relates them to other problems. 

Proposition \ref{Fform} states that if a Lie algebra is defined over a field $F$, then every conjugate Lie algebra over the field $F$ is isomorphic. In the special case of the field extension $\R \subseteq \C$, this is stating that if the complex Lie algebra is defined over $\R$, then $\lie \approx \overline{ \lie }$. We conjecture that the converse holds as well.

\begin{Con}
	\label{conRC}
If $\lie$ is a complex Lie algebra such that $\lie \approx \overline{\lie}$, then $\lie$ is defined over $\R$.
\end{Con}
\noindent For Lie algebras of dimension $\leq 4$ the conjecture holds, see Proposition \ref{conjfor4}. We do not know of any method which shows that $\lie$ is not defined over $\R$ which does not imply that in fact $\lie \not \approx \overline{ \lie}$ as in Proposition \ref{onlydiffconj}.

We now state an equivalent conjecture about the existence of real forms in direct sums. If the complex Lie algebra $\lie$ is defined over $\R$, then also the direct sum $\lie \oplus \lie$ is defined over $\R$. The conjecture whether the converse of this statement always holds is equivalent to Conjecture \ref{conRC}.

\begin{Con}
	\label{realsum}
Let $\lie$ be a complex Lie algebra such that $\lie \oplus \lie$ is defined over $\R$, then also $\lie$ is defined over $\R$.
\end{Con}

\begin{proof}[Proof of equivalence between Conjecture \ref{conRC} and \ref{realsum}:]
First assume that Conjecture \ref{conRC} is true and assume that $\lie$ is a complex Lie algebra with $\lie \oplus \lie$ defined over $\R$. By Proposition \ref{Fform} we get that $\lie \oplus \lie \approx \overline{ \lie \oplus \lie }  $ or hence that $$\lie \oplus \lie \approx \overline{\lie} \oplus \overline{ \lie}$$ by Lemma \ref{sumsigma}. By Proposition \ref{copies} we conclude that $\lie \approx \overline{\lie}$ or thus that $\lie$ has a real form by assumption. 

Next we assume that Conjecture \ref{realsum} holds and that $\lie$ is a complex Lie algebra with $\lie \approx \overline{ \lie }$. By Theorem \ref{sumconjugate} we get that $\lie_{\R}$ is a real form of $\lie \oplus \overline{ \lie } \approx \lie \oplus \lie$. Hence by assuming Conjecture \ref{realsum} we find that $\lie$ is defined over $\R$.
\end{proof}

The generalized version of Conjecture \ref{conRC} for arbitrary field extensions is widely open, since we only dealt with fields $E$ of degree $2$ over $\Q$ before.
\begin{QN}
	\label{Qcon}
Let $F \subseteq E$ be a Galois extension and $\lie$ a Lie algebra over $E$. Assume that $\lie^\sigma \approx \lie$ for all $\sigma \in \Gal(E,F)$, does it hold that $\lie$ is defined over $F$? 
\end{QN}
\noindent A positive answer to Question \ref{Qcon} would relate the existence of $F$-forms to studying conjugate Lie algebras, leading to new directions for studying Lie algebras over different fields.

\label{sec:qu}

\bibliography{ref}
\bibliographystyle{plain}

\end{document}